%% file: ReductionR4arxiv.tex
\documentclass[11pt]{amsart}
\usepackage{geometry}                % See geometry.pdf to learn the layout options. There are lots.
\usepackage{graphicx}
\usepackage{amssymb}
\usepackage{epstopdf}
\newtheorem{Lemma}{Lemma}
\newtheorem{Theorem}{Theorem}

\newcommand{\R}{\mathbb{R}}
\newcommand{\eff} { \mathrm{eff} }

\title{Symmetry reduction of the 3-body problem in $\R^4$}
\author{Holger R.~Dullin, J\"urgen Scheurle}
\date{}                                           % Activate to display a given date or no date

\begin{document}

\begin{abstract}
The 3-body problem in $\R^4$ has 24 dimensions and is invariant under translations and rotations. 
We do the full symplectic symmetry reduction and obtain a reduced Hamiltonian in local symplectic coordinates
on a reduced phase space with 8 dimensions. The Hamiltonian depends on two parameters $\mu_1 > \mu_2 \ge 0$, 
related to the conserved angular momentum. The limit $\mu_2 \to 0$ corresponds to the 3-dimensional limit.
We show that the reduced Hamiltonian has relative equilibria that are local minima and hence Lyapunov stable when 
$\mu_2$ is sufficiently small. 
This proves the existence of balls of initial conditions of full dimension that do not contain any orbits that are unbounded.
\end{abstract}

 \dedicatory{Dedicated to James Montaldi}

\maketitle

\section{Introduction}

Consider $N$ masses $m_i$ at positions $r_i \in \R^d$, $i = 1, \dots, N$, 
moving under the influence of Newtonian attraction with potential 
\[
   U = -\sum_{1 \le i < j \le N} \frac{m_i m_j}{ || r_i - r_j||} ,
\]
so that Newton's equations of motion are 
\[
   m_i \ddot r_i  = - \nabla_{r_i} U, \quad i = 1, \dots, N.
\]
Here ${|| \cdot  ||}$ denotes the Euclidean norm on $\R^d$. These equations are invariant under translations and Galilein boosts $r_i \to r_i + c + v t$ for 
some constant vectors $c, v \in \R^d$ and under rotations $r_i \to M r_i$ for some 
constant matrix $M \in SO(d)$. The corresponding conserved quantities are the total 
linear momentum $\sum m_i \dot r_i$ and the total angular momentum
$\sum m_i r_i \wedge \dot r_i$. In addition there is a scaling symmetry $r_i \to s r_i$ and
$t \to t s^{3/2}$ for constant scalar $s$.

The goal of this paper is to reduce these equations by translation and rotation  symmetry, specifically 
for $N=3$ and $d=4$. The cases $d = 1, 2, 3$ have been studied extensively in the 
classical literature, see, e.g., \cite{Whittaker37}. Larger $d$ have more recently been studied 
by Albouy and Chenciner \cite{AlbouyChenciner98} and Chenciner \cite{Chenciner11}.
For $N=3$ the case $d=4$ is interesting because new dynamics appears compared to $d=3$.
In particular the balanced configurations introduced in \cite{AlbouyChenciner98} are relative 
equilibria that do not exist for $d=3$. 
For $N=3$, cases with $d>4$ do not, by contrast, produce new dynamics compared to $d=4$.
The reduction we are going to present holds for an arbitrary potential that 
depends on the distances $||r_i - r_j||$ only. It is based on a novel approach to the well-known procedure 
of eliminating the nodes which dates back to Jacobi \cite{Jacobi1843} in the case $d=3$. 
% Deprit 1983, Radau?

With the fully reduced Hamiltonian function (Hamiltonian) at hand it is then straightforward to find new relative 
equilibria and analyse their stability. Our second main theorem shows that there is a family of 
relative equilibria that corresponds to minima of the Hamiltonian, and thus constitute solutions of the 
3-body problem that are Lyapunov stable. A simple corollary  is that  M.~Herman's 
``Oldest problem in dynamical systems'' on whether the set of unbounded solutions 
is dense for negative energy \cite{Herman98} can be answered  for $d=4$ by ``No!": 
There is a ball of full dimension that does not contain any unbounded solutions. 
% Clearly this result is specific to $d=4$.

The realisation that the 3-body problem in $\mathbb{R}^4$ has Lyapunov stable relative equilibria
was conceived in discussions with Alain Albouy, Rick Moeckel, James Montaldi and Alain Chenciner 
at the Observatory in Paris in 2015.
Some of the results of these discussions are presented in the preprint \cite{AD19}.
In \cite{AD19} is it shown that there is a global minimum of the Hamiltonian for generic angular momentum, 
and some properties of the families of relative equilibria are proved. 
By contrast, in the present paper we prove that all three families of relative equilibria are minima 
when the angular momentum is sufficiently close to the (non-generic) 3-dimensional case.

After this paper was finished a related preprint \cite{SS19} appeared. 
In that preprint only the subgroup $SO(2)\times SO(2)$ of the full rotational symmetry group $SO(4)$ is 
considered in the reduction and hence  somewhat different results are obtained.

\section{Translation Reduction}

Translation reduction is well known and can be achieved by introducing Jacobi vectors. 
Define vectors $x_i \in \R^d$ by
\[
    x_1 =  r_2 - r_3, \quad
    x_2 =  r_1 - \frac{m_2 r_2 + m_3 r_3}{m_2 + m_3} , \quad
    x_3 = \frac{m_1 r_2 + m_2 r_2 + m_3 r_3}{m_1 + m_2 + m_3}
\]
and conjugate momenta $y_i \in \R^d$ by 
\[
   y_1 = \frac{-m_3\dot r_2 + m_2 \dot r_3}{m_2+m_3}, \quad
   y_2 =  \frac{ (m_2 + m_3) \dot r_1 + m_1 \dot r_2 + m_1 \dot r_3 }{m_1+m_2+m_3} , \quad 
   y_3 = m_1 \dot r_1 + m_2 \dot r_2 + m_3 \dot r_3 \,.
\]
Clearly $x_3$ is the centre of mass and $y_3$ is the total linear momentum,
both of which are set to zero from now on.

The mutual distances in these coordinates become
\[
   ||r_2 - r_3|| = ||x_1||, \quad
   ||r_3 - r_1|| =  || a_2 x_1  + x_2 ||, \quad
   ||r_1 - r_2|| = || a_3 x_1- x_2|| 
\]
with $a_i = m_i/(m_2 + m_3)$, $i = 2,3$, so that the potential is a function
of the scalar products $x_i \cdot x_j$, $i, j = 1, 2$ only.

Define the reduced masses
\[
	\nu_1 = \frac{m_2m_3}{m_2+m_3}, \quad
	\nu_2 = \frac{m_1(m_2+m_3)}{m_1+m_2+m_3}
\]
so that the translation reduced Hamiltonian becomes
\begin{equation} \label{eqn:Hxy}
   H = \frac{1}{2\nu_1}  ||y_1||^2 + \frac{1}{2\nu_2}|| y_2 ||^2 + V( ||x_1||^2 , ||x_2||^2, x_1 \cdot x_2)
\end{equation}
with $x_1, x_2, y_1 , y_2 \in \R^d$. 

The Hamiltonian \eqref{eqn:Hxy} is invariant under rotations
$(x_1, x_2, y_1, y_2) \to ( M x_1, Mx_2, My_1, My_2)$.
The corresponding angular momentum is given by the angular momentum 
\[
    L = x_1 \wedge y_1 + x_2 \wedge y_2 \in so(d)
\]
which for $d=4$ has 6 independent components. The wedge product 
can be expressed as an anti-symmetric matrix using $x \wedge y = x y^t - y x^t$.
Since $L$ is anti-symmetric, for $d=4$
its characteristic polynomial is even and has two invariants:
The determinant of $L$ which is a perfect square called the Pfaffian of $L$,
and the trace of $L^2$. Denote the eigenvalues of $L$ by $\pm i \mu_1$ and
$\pm i \mu_2$, so that $\mathrm{Pf}(L) = \mu_1 \mu_2$ and $\mathrm{tr} L^2 = -2(\mu_1^2 + \mu_2^2)$.

\section{Rotation Reduction}

Reduction by rotations does depend on the dimension $d$. 
In order to generalize the reduction procedure due to Jacobi as, e.g., described in \cite{Whittaker37},
to the case $d=4$, we introduce a basis for the plane in $\R^4$ spanned by the two vectors $x_1$ and $x_2$
through an orthogonal rotation matrix $M$. In this basis we can write 
\[
     x_1 = M  q_{12}, \, q_{12} = (q_1, q_2, 0, 0)^t, \quad x_2 = M  q_{34}, \, q_{34} = (q_3, q_4, 0, 0)^t \,.
\]
Since the potential is a function of the scalar products $x_i \cdot x_j$, 
in the new coordinates the potential depends only on $q_1^2 + q_2^2$, $q_3^2 + q_4^2$, 
$q_1 q_3 + q_2 q_4$. The main problem is to determine the form of the kinetic energy in the 
new coordinates.

Define two essential quantities: The oriented area $A$ of the triangle formed by two vectors $x_1$ and $x_2$ in configuration space and an angular momentum $L_3$ as, respectively, 
\[
     A = \tfrac12 (q_1 q_4 - q_2 q_3), \quad
     L_3 = q_1 p_2 - q_2 p_1 + q_3 p_4 - q_4 p_3 \,.
\]
The orthogonal matrix $M$ is a product of elementary rotation matrices.
Since $(q_1, q_2, q_3, q_4)$ already give 4 degrees of freedom, the rotation $M$ 
needs to have another 4 degrees of freedom, say $(\psi_1, \psi_2, \theta_1, \theta_2)$. 
Notice that not all of the 6 dimensions of $SO(4)$ are used in this way. 
% However, we just need to rotate a given 2-plane into 
% The new coordinates are $q_{new} = (q_1, q_2, q_3, q_4, \psi_1, \psi_2, \theta_1, \theta_2)$.

Denote a basis of $so(4)$ by 
$B_{ij} = E_{ij} - E_{ji}$, where $E_{ij}$ is the matrix with all entries equal to zero except for the $ij$-entry which is 1.
Now define the rotation matrix $M \in SO(4)$ by
\[
    M = \exp(B_{12} \theta_1 ) \exp(B_{34}  \theta_2 ) \exp( B_{13} \psi_1 ) \exp( B_{24} \psi_2 )  = M_{\theta} M_{\psi}\,.
\]
Notice that the first two factors and the last two factors commute.

\begin{Lemma}
A symplectic transformation from $(x_1, x_2, y_1, y_2) \in \R^{16}$ to 
new local coordinates  $q_{new} = (q_1, q_2, q_3, q_4, \psi_1, \psi_2, \theta_1, \theta_2)$ and 
momenta $p_{new} = (p_1, p_2, p_3, p_4,  p_{\psi_1}, p_{\psi_2}, p_{\theta_1}, p_{\theta_2})$ is given by 
\begin{equation} \label{eqn:sympl}
   x_1 = M \begin{pmatrix} q_1 \\ q_2 \\ 0 \\ 0 \end{pmatrix}, \quad
   x_2 = M \begin{pmatrix} q_3 \\ q_4 \\ 0 \\ 0 \end{pmatrix}, \quad
   y_1 =  M \begin{pmatrix}  p_1 \\   p_2 \\ \alpha_1 \\ \alpha_2 \end{pmatrix}, \quad
   y_2 =  M \begin{pmatrix}  p_3 \\ p_4 \\ \alpha_3 \\ \alpha_4 \end{pmatrix}
\end{equation}
where $\alpha_i$ are linear in all momenta and given by
\[
\alpha_1 =    q_3 B  -  \frac{q_4 p_{\psi_1} }{2A}, \quad
\alpha_2 =   -q_4C + \frac{q_3  p_{\psi_2}}{2A}, \quad
\alpha_3 =   - q_1 B + \frac{q_2  p_{\psi_1}}{2A}, \quad
\alpha_4 =   q_2 C  - \frac{q_1  p_{\psi_2}}{2A}
\]
\begin{align*}
B &= \frac{   L_3 \sin 2\psi_1 + 2(p_{\theta_1} \sin\psi_1 \cos\psi_2 + p_{\theta_2} \cos\psi_1 \sin\psi_2)  }{ 2 A (\cos2\psi_1 - \cos2\psi_2) } \\
C &= \frac{   L_3 \sin 2\psi_2 + 2(p_{\theta_1} \cos\psi_1 \sin\psi_2 + p_{\theta_2} \sin\psi_1 \cos\psi_2)  }{ 2 A (\cos2\psi_1 - \cos2\psi_2) } 
\end{align*}
\end{Lemma}
\begin{proof}
%This can be proved using a generating function $W$ that is a function of new coordinates and old momenta
%\[   W = q_{12}^t M y_1 + q_{34}^t M y_2 \]
% so that by construction the new coordinates are as desired and the momenta transform as
%\[
%      q_{old} = \partial W / \partial p_{new}
%     p_{new} = \frac{ \partial W } { \partial q_{new} } \,.
%\]
In configuration space, the old coordinates are expressed as functions of the new coordinates, 
$q_{old} = F ( q_{new} )$. This is extended to a canonical symplectic transformation by 
cotangent lift $p_{old}  = (DF)^{-t}  p_{new}$. In our case there is a special 
structure because in front of the vectors on the right-hand sides of the relations in \eqref {eqn:sympl} we have an orthogonal matrix as a prefactor. 

%\[     p_{new} =  \left( \frac{  \partial ( M q_{12}, M q_{34} )   }{ \partial  q_{new} } \right)^t p_{old} \]
%The derivative with respect to $(q_1, q_2, q_3, q_4)$ produces the first two columns of $M$.
%The derivatives with respect to the angles are more complicated, and depends on the precise 
%form of $M$. 

The derivative of $x_1$ is given by
\[
   \frac{ \partial x_1 }{\partial q_{new}} = M \begin{pmatrix} 
   1 & 0 & 0 & 0 & 0 & 0 & q_2 \cos\psi_1\cos\psi_2 & q_2 \sin\psi_1\sin\psi_2  \\
   0 & 1 & 0 & 0 & 0 & 0 & -q_1 \cos\psi_1\cos\psi_2 & -q_1 \sin\psi_1\sin\psi_2  \\
   0 & 0 & 0 & 0 & -q_1 & 0 & q_2 \sin\psi_1\cos\psi_2 & -q_2 \cos\psi_1\sin\psi_2  \\
   0 & 0 & 0 & 0 & 0 &  -q_2 & -q_1  \cos\psi_1\sin\psi_2 & q_1 \sin\psi_1\cos\psi_2    
   \end{pmatrix} = M U_{12} \\\ ,
\]
and the derivative of $x_2$ is
\[
   \frac{ \partial x_2  }{\partial q_{new}} = M \begin{pmatrix} 
   0 & 0 & 1 & 0 & 0 & 0 & q_4 \cos\psi_1\cos\psi_2 & q_4 \sin\psi_1\sin\psi_2  \\
   0 & 0 & 0 & 1 & 0 & 0 & -q_3 \cos\psi_1\cos\psi_2 & -q_3 \sin\psi_1\sin\psi_2  \\
   0 & 0 & 0 & 0 & -q_3 & 0 & q_4 \sin\psi_1\cos\psi_2 & -q_4  \cos\psi_1\sin\psi_2  \\
   0 & 0 & 0 & 0 & 0 & -q_4 & -q_3  \cos\psi_1\sin\psi_2 & q_3 \sin\psi_1\cos\psi_2  
   \end{pmatrix} = M U_{34} \,.
\]
The non-trivial entries can be computed from $M^t \dot M$. For example the 
last column of $U_{12}$ is given by $ M^t ( \partial M / \partial \theta_2 )q_{12}  $.
Forming the matrix $U^t = ( U_{12}^t , U_{34}^t ) $ the hard work is to invert $U$. 
The determinant of $U$  is 
\[
    \det U = \frac12 A^2 ( \cos 2\psi_1 - \cos 2 \psi_2) \,.
\]
Now the cotangent lift is given by  $p_{old} = \mathrm{diag}( M , M ) U^{-t} p_{new}$ and 
this gives  the formulas claimed.
% \footnote{maybe simpler using a generating function? Or by direct verification of the symplectic condition? Maybe using 2-form?}
\end{proof}

\begin{Lemma}
In the new coordinates the Hamiltonian becomes
\begin{equation} \label{eqn:Hpsi}
   H = \frac12 ( p_1^2 + p_2^2 + \tilde f(q_3, q_4) ) + \frac12( p_3^2 + p_4^2 + \tilde f(q_1, q_2)) + V( q_1^2+q_2^2, q_3^2 + q_4^2, q_1 q_3 + q_2 q_4)
\end{equation}
where
\[
    \tilde f(q_i, q_j)  = \left( q_i B - \frac{ q_j p_{\psi_1} }{ 2 A} \right)^2 + \left(-q_j C + \frac{ q_i p_{\psi_2}}{  2 A } \right)^2 \,.
\]
%    \frac{1}{A^2} \left( g( q_i, q_j) + \frac{ q_i q_j}{ \sin\delta \sin\sigma }  ( c_{1} p_{\psi_1} + c_{2} p_{\psi_2})  + q_i^2 p_{\psi_1}^2  + q_j^2 p_{\psi_2}^2\right)\\
%     c_{1} & = \Sigma \sin\sigma  + \Delta \sin\delta  + L_3 \sin 2 \psi_1  \\ 
 %    c_{2} & = \Sigma \sin\sigma  - \Delta \sin\delta  - L_3 \sin 2 \psi_2   \,.
%   w_1^2 &= \frac{1}{A^2} \left( c_{10} + \frac{ q_3 q_4}{ \sin\delta \sin\sigma }  ( c_{11} p_{\psi_1} + c_{12} p_{\psi_2})  + q_4^2 p_{\psi_1}^2  + q_3^2 p_{\psi_2}^2\right) ,\\
%   w_2^2 &= \frac{1}{A^2} \left( c_{20} + \frac{ q_1 q_2}{ \sin\delta \sin\sigma }  (c_{11}  p_{\psi_1} + c_{12} p_{\psi_2}) + q_2^2 p_{\psi_1}^2  + q_1^2 p_{\psi_2}^2\right) ,
In particular, $H$ is independent of $\theta_1$ and $\theta_2$.
\end{Lemma}
\begin{proof}
This is a simple consequence of the previous Lemma.
\end{proof}
Notice that $H$ can be considered as a partially reduced Hamiltonian function with two parameters
$p_{\theta_1}  =\mu_1$ and $p_{\theta_2} = \mu_2$ and two cyclic angles $\theta_1$ and $\theta_2$. 
This reduced Hamiltonian has 6 degrees of freedom and a 12-dimensional phase space $M^{12}$.

\begin{Lemma}
In the new  coordinates  the angular momentum $L$ satisfies 
\[
   M_\theta^t  L M_\theta = -B_{12} p_{\theta_1} - B_{34} p_{\theta_2} - B_{13} p_{\psi_1}  - B_{24} p_{\psi_2} 
   + \frac{1}{2 \sin\delta \sin\sigma } \left( B_{23} (u_1+u_2) + B_{14} (u_1-u_2)\right)
\]
where
\[
   u_1 = L_3 - \Sigma \cos \delta, \quad
   u_2 = L_3 - \Delta \cos \sigma
%    L_3 = q_1 p_2 - q_2 p_1 + q_3 p_4 - q_4 p_3
\]
and 
\[
   \sigma = \psi_1 + \psi_2, \quad
   \delta = \psi_1 - \psi_2, \quad
    \Sigma = p_{\theta_1} + p_{\theta_2}, \quad
   \Delta = p_{\theta_1} - p_{\theta_2}
\]
\end{Lemma}
\begin{proof}
Note that for orthogonal $M$ we have
\[
   M x \wedge M y = Mx (M y)^t - My (Mx)^t = M ( x^t y - y^t x) M^t ,
\]
which says that the momentum map $L$ is equivariant with respect to the rotation given by $M$.
Using $M = M_\theta M_\psi$ then gives
%\[
%    L  =  M_\theta M_\psi ( q_{12} \wedge p_{12} +  q_{34} \wedge  p_{34} ) M_\psi^t M_\theta^t
%\]
%such that 
\[
   M_\theta^t L M_\theta = M_\psi \hat L M_\psi^t , \quad \hat L =  q_{12} \wedge p_{12} +  q_{34} \wedge  p_{34} ,
\]
where $\hat L$ is the angular momentum tensor in the body frame defined by $M$.
Here $q_{ij}$ and $p_{ij}$ refer to the vectors on the right-hand sides of the relations in \eqref{eqn:sympl}.
It is straightforward to compute
\[
    \hat L = B_{12} L_3  
     + B_{13} (q_1 \alpha_1 +q_3 \alpha_3) + B_{14} (q_1 \alpha_2 + q_3 \alpha_4) 
    + B_{23}(q_2 \alpha_1 + q_4 \alpha_3) + B_{24} (q_2 \alpha_2 + q_4 \alpha_4) \,.
\]
Using the definitions of $\alpha_i$ the coefficient of $B_{13}$ reduces to $-p_{\psi_1}$
and the coefficient of $B_{24}$ reduces to $-p_{\psi_2}$. Conjugating $\hat L$ 
with $M_\psi$ gives the result.
\end{proof}

So far this is a partial reduction: We introduced two cyclic angles $\theta_1$ and $\theta_2$ 
with conjugate momenta that are now constants of motion. However, this is a reduction by 
two degrees of freedom only, i.e.{} to 6 degrees of freedom, but we would like to reduce by another two degrees of freedom (elimination of the nodes), 
so that the reduced system has 4 degrees of freedom. Here and subsequently, we assume the values $\mu_1$ and $\mu_2$ of the momenta $p_{\theta_1}$ and $p_{\theta_2}$, respectively, to be fixed and generic. 

Symplectic symmetry reduction in the abstract is described by a fundamental theorem by 
Marsden and Weinstein  \cite{MW74}.
According to that, one fixes a regular value $\mu$ of a momentum map which is 
supposed to be equivariant with respect to the symmetry group, and then takes 
the quotient of the corresponding level set of the momentum map by the isotropy subgroup of $\mu$. 
Provided that the isotropy subgroup acts 
%We have an action of the compact Lie group $G=SO(4)$ with an equivariant momentum mapping $L$.
%Fixing are regular value $\mu$ of $L$ the reduced symplectic manifold is $L^{-1}(\mu) / G_\mu$
%assuming that  the isotropy group $G_\mu$ ... %$L^{-1}(\mu)$.
%
acts freely and properly on that level set, the quotient defines the reduced symplectic manifold
with a unique symplectic form. 
For commutative groups both steps reduce by the same dimension. For the case of $SO(d)$ 
the isotropy group that fixes a generic element of $so(d)$ has dimension $[ d/2 ]$, which is the 
dimension of % the Cartan-subalgebra of $so(d)$. 
% It is given by 
the number of real invariant 2-dimensional eigenspaces of $L$,
and corresponds, for $d=4$ to our two cyclic angles $\theta_1$ and $\theta_2$.
In the 3-body problem % the action of $G_\mu$ is not free, since 
collinear configurations with zero angular momentum 
% momenta pointing in the same direction 
are fixed by a continuous group of rotations.
% more rotations than the generic planar configurations.
Hence, the value of the momentum map is not regular there, 
and we expect the corresponding reduced space to have 
a singularity there. 
%
% o have a the reduced space to be globally smooth.
% This is the reason 
In fact, the symplectic coordinates which we have introduced in the present paper  
are valid near generic planar configurations only.

\section{Restriction to  an 8-dimensional invariant subset}

We are now choosing a coordinate system in which $L$ has a particularly simple form, which is adapted to our choice of $M$. 
Let $L$  be equal to
\[
    L_0= \begin{pmatrix} 
        0 & \mu_1 & 0 & 0 \\
        -\mu_1 & 0 & 0 & 0 \\
        0 & 0 & 0 & \mu_2 \\
        0 & 0 & -\mu_2 & 0 
    \end{pmatrix} = \mu_1 B_{12} + \mu_2 B_{34}
\]
Notice that since $L_0$ is spanned by $B_{12}$ and $B_{34}$ and since 
these two matrices commute, we conclude that we have  $ M_\theta^t L_0  M_\theta = L_0$. 
Thus the two-parameter subgroup of matrices $M_\theta$ is the isotropy group that fixes $L_0$.
%
% \footnote{ mention Cartan and Weyl...?}
Our approach was inspired by the treatment of the case $d=3$ in Whittaker \cite{Whittaker37}.

\begin{Lemma} \label{lem4}
The subset $\mathcal{I} \subset M^{12}$ defined as the zero-level of the map $\mathcal{C} : M^{12} \to \R^4$ with components $c_i$ 
%  of the 12-dimensional phase space with coordinates  $( q_1, q_2, q_3, q_4, \psi_1, \psi_2, p_1, p_2, p_3, p_4, p_{\psi_1}, p_{\psi_2}) $ 
 defined by
\begin{equation} \label{eqn:Invset}
% \mathcal{I} = \{ ( q_1, q_2, q_3, q_4, \psi_1, \psi_2, p_1, p_2, p_3, p_4, p_{\psi_1}, p_{\psi_2}) \,:\,
    c_1 = p_{\psi_1} , \quad c_2 = p_{\psi_2}, \quad 
    c_3 = \Sigma \cos \delta   - L_3, \quad c_4 =  \Delta  \cos \sigma - L_3
%    \}
\end{equation}
% where $L_3 =  q_1 p_2 - q_2 p_1 + q_3 p_4 - q_4 p_3$ 
is an invariant set with respect to the Hamiltonian flow of the partially reduced Hamiltonian \eqref{eqn:Hpsi} 
for fixed values 
\[
p_{\theta_1} = \mu_1, \quad 
p_{\theta_2} = \mu_2.
\]
Locally $\mathcal{I}$ is an 8-dimensional manifold almost everywhere (near any regular point with respect to the map $\mathcal{C}$). 
\end{Lemma}
\begin{proof}
Since $L$ is constant we can impose $L = L_0$. Combined with the previous 
Lemma on $L$ this implies that $p_{\theta_1} = \mu_1$, $p_{\theta_2} = \mu_2$, 
and four additional equations that together determine 
the six components of $L$. After a little bit of algebra we see that these 
imply $c_i = 0$ as stated in the Lemma. Of course it is also possible to check that 
the set $\mathcal{I}$ is invariant by directly using the Hamiltonian vector field corresponding to the partially reduced Hamiltonian
$H$ in \eqref{eqn:Hpsi}.
\end{proof}

We next show that it is possible to obtain the reduced Hamiltonian (near any regular point of the invariant set $\mathcal{I}$) by simply restricting the Hamiltonian $H$ in \eqref{eqn:Hpsi} to $\mathcal{I}$. This is a consequence of the following general Theorem.
%Whittakers argument is combining the chain rule with the fact that on the invariant set we have
%$\dot p_{\psi_i} = 0$, since $p_{\psi_i} = const$. The chain rule produces terms $\partial H/ \partial \psi_i$,
%but by Hamiltons equation this is $\dot p_{\psi_i}$, and hence vanishes.

\begin{Theorem}
Let $(M,\omega)$ be a smooth, 2d-dimensional symplectic manifold equipped with some symplectic form $\omega$, and let $N = \mathcal{C}^{-1}(c) \subset M$ be a smooth, (2d - k)-dimensional submanifold ($k, d\in \mathbb N$, $0 < k < d$), where $c \in \R^k$ is a regular value of the smooth map 
\[
\mathcal{C} : M \to \R^k; \; z  \mapsto \left( \begin{array}{c} c_1(z) \\ \vdots \\ c_k(z) \end{array}
\right) .
\]
Furthermore, assume the square matrix % of commutators 
\[
\mathcal{A}(z) := \big( \omega (z) (X_{c_i}(z), X_{c_j}(z)) \big)_{i,j = 1,  ...  ,k}
\]
to be regular for all $z \in N$. Here, for any smooth function $g $, $X_g : M \to TM$ denotes the corresponding Hamiltonian vector field which is uniquely defined by
\[
\omega(z)(X_g(z), w) = dg(z)w
\]
for all $z \in M$, $w \in T_{z}M$. Then $(N,\omega_{|N})$ is  a symplectic submanifold of $(M,\omega)$. Note, accordingly, $k$ is even. 

In addition, consider a Hamiltonian dynamical system 
\begin{equation} \label{eqn:Haml}
\dot{x} = X_H(z), \; z \in M
\end{equation}
corresponding to some smooth Hamiltonian $H $, and suppose that $N$ is invariant under the flow of that system, i.e.
\[
X_H(z) \in T_{z}N
\]
for all $z  \in N$. Then 
\[
(X_H)_{|N} = X_{H_{|N} }
\]
is the Hamiltonian vector field of the reduced (restricted) system on $(N,\omega_{|N})$. 
\end{Theorem}

\begin{proof} 
To prove that $(N, \omega_{|N})$ is a symplectic submanifold of $(M,\omega )$ we first show that $\omega_{|N}$ is a symplectic form on $N$. In fact,
\[
V_z := span \{X_{c_1}, ... , X_{c_k} \} \subset T_{z}M
\]
equipped with the symplectic form $\omega(z)_{|V_z}$ defines a k-dimensional symplectic subspace of the symplectic vector space $(T_{z}M, \omega(z))$ for all $z \in N$. This is a straightforward consequence of the regularity of the matrix $\mathcal{A}(z)$. In turn, the $\omega(z)$-orthogonal, complementary subspace of $V_z$ in $T_{z}M$,
\[
{V_{z}}^{\omega} := \{v \in T_{z}M | \, \omega(z)(v,w) = 0 \;  \text{ for all } \;w \in V_z \}
\]
is symplectic, too; $T_{z}M = V_z \oplus {V_{z}}^{\omega}$. But $V_z \subset (T_{z}N)^{\omega}$, since by definition
\[
\omega(z)(X_{c_i}(z), w) = d{c_i}(z)w =0
\] for all $i$ and all $w \in T_{z}N$. So, $T_{z}N \subset  {V_{z}}^{\omega}$. Moreover,
\[
dim \, T_{z}N = dim \, TM - k = dim \, TM - dim \, V_z = dim \, {V_{z}}^{\omega},
\]
i.e. $T_{z}N = {V_{z}}^{\omega}$ is a symplectic subspace of $(T_{z}M, \omega(z))$ for all $z \in N$. In particular, this implies that $\omega_{|{T_{z}N}}$ is a symplectic form on $T_{z}N$ for all $z \in N$.  Also, $\omega_{|N}$ is a closed differential form on the manifold $N$, since $d(\omega_{|N}) = (d\omega)_{|N} = 0$. Therefore, $\omega_{|N}$ satisfies all the axioms to be fulfilled for a symplectic form on $N$. 
\\

Now, we proceed to prove the second assertion of Theorem 1. Since $(N, \omega_{|N})$ is a symplectic manifold, there exists a unique vector field $X_{H_{|N}}$ on $N$ such that 
\[
\omega(z) (X_{H_{|N}}(z),w) = d(H_{|N})(z)w
\] 
for all $z \in N$ and all $w \in T_{z}N = {V_z}^{\omega}$. But we also have 
\[
\omega(z)(X_{H}(z),w) = dH(z)w = d(H_{|N})(z)w
\]
for all $z \in N$ and all $w \in T_{z}N = {V_z}^{\omega} \subset T_{z}M$, where $X_{H}(z) \in T_{z}N$ holds by the invariance property assumed to be satisfied for $N$. Therefore, uniqueness of the vector field $X_{H_{|N}}$ on $N$ implies $(X_H)_{|N} = X_{H_{|N}}$. 

\end{proof}

So, in order for the invariant set $\mathcal{I}$ described in Lemma~\ref{lem4} to be a symplectic submanifold almost everywhere, 
we need to check that  the  matrix $\mathcal{A}(z)$ defined in the previous Theorem is non-singular.
There are 4 non-zero entries.
%  between $c_i$, $i = 1,2,3,4$.
The determinant of the matrix  is 
% \footnote{we know the commutators, because they are angular momenta, use this!?}
\[
     \frac{(\Sigma + L_3 \cos \delta)(\Delta + L_3 \cos\sigma)}{\sin^4\sigma \sin^4\delta}
\]
Restricted to the invariant set this simply becomes $\Delta^2 \Sigma^2 = (\mu_1^2 - \mu_2^2)^2$ and hence 
is non-vanishing as long as $|\mu_1| \not = |\mu_2|$. In fact, even though the matrix $\mathcal{A}(z)$
is somewhat complicated, when restricted to the invariant set $\mathcal{I}$ the only non-zero entries 
are $\pm p_{\theta_i}$.
%  \footnote{maybe this is another lemma?}

\begin{Lemma} \label{lem5}
The functions $\tilde f$ in  \eqref{eqn:Hpsi} restricted to the invariant set $\mathcal{I}$ defined in  \eqref{eqn:Invset}
are given by  
\[
     \tilde f|_\mathcal{I} = f(q_i, q_j, L_3)  = \frac{1}{16 A^2} \left(  (L_d + L_s)^2 q_i^2 + ( L_d - L_s)^2 q_j^2 \right) \\
%    w_1^2 = \frac{1}{16 A^2} ( (L_d + L_s)^2 q_3^2 + ( L_d - L_s)^2 q_4^2)  \quad
%    w_2^2 = \frac{1}{16 A^2} ( (L_d + L_s)^2 q_1^2 + (L_d - L_s)^2 q_2^2) 
\]
where $L_d$ and $L_s$ are functions of $L_3$ and the constants 
$\Delta = \mu_1 - \mu_2$ 
 and 
 $\Sigma =  \mu_1 + \mu_2$:
\[
   L_d^2 = \Delta^2 - L_3^2, \quad
   L_s^2 = \Sigma^2 - L_3^2
%   L_3^2 = q_1 p_2 - q_2 p_1 + q_3 p_4 - q_4 p_3 \,.
\,.
\]
\end{Lemma}
\begin{proof}
Setting $p_{\psi_i} = 0$ and rewriting $f$ in terms of $\sigma$ and $\delta$ gives
\begin{align*}
     B &= -\frac{ (L_3 \cos\sigma + \Delta)/\sin\sigma + ( L_3 \cos\delta + \Sigma)/\sin\delta }{4A }, \\
     C &= \frac{ (L_3 \cos\sigma + \Delta)/\sin\sigma - ( L_3 \cos\delta + \Sigma)/\sin\delta }{4A } \,.
\end{align*}
Now inserting the definitions of the invariant set \eqref{eqn:Invset} reduces $\tilde f$ to $f$.
\end{proof}

This leads us to our first main result.
\begin{Theorem}
The fully reduced Hamiltonian of the 3-body problem in 4-dimensional space 
with constant angular momentum matrix with eigenvalues $\pm i \mu_1, \pm i \mu_2$ , $\mu_1 > \mu_2 \ge 0$
in local coordinates $q_i, p_i$, $i = 1, \dots, 4$ assuming $A = \tfrac12 ( q_1 q_4 - q_2 q_3) \not = 0$ 
is given by 
\begin{equation} \label{eqn:Hamred}
H = \frac{1}{2\nu_1} ( p_1^2 + p_2^2 +  f(q_3, q_4) ) + \frac{1}{2\nu_2}( p_3^2 + p_4^2 +  f(q_1, q_2)) + V( q_1^2+q_2^2, q_3^2 + q_4^2, q_1 q_3 + q_2 q_4)
\end{equation}
where $f$ is defined in Lemma~\ref{lem5}. 
\end{Theorem}

Note that the old momenta are linear in the new momenta, and hence the kinetic energy is 
homogeneous of degree 2 in momenta.
After restricting to the invariant set, however, the kinetic energy is not polynomial
in the momenta, even though it still is homogeneous 
of degree 2 in momenta (including the constant momenta $p_{\theta_i}$).

We can introduce polar coordinates in the plane with vectors
$(q_1, q_2)$ and $(q_3, q_4)$. The momentum conjugate to the corresponding angle 
will be the angular momentum $L_3$. However, the terms in $\tilde f$ are
not rotationally symmetric, so this angle will not be cyclic. Introducing this 
angle would make explicit the separation into shape coordinates and 
orientation coordinates.

\section{Effective Potential}

According to Smale's program \cite{smale70two} finding relative equilibria is reduced 
to finding critical points of an effective potential after reduction. 
It is interesting to note that for $d=2,3$ this approach works nicely, since 
the Hamiltonian is quadratic in momenta, and additional terms can be 
considered to be part of the potential. Linear terms in momenta can be 
considered as effective magnetic fields.
For $d=4$ the reduced Hamiltonian is, however, not quadratic in momenta,
and thus defining an effective potential in the usual way is not possible. 
However, we are interested in relative equilibria with vanishing $p$.

%The  kinetic energy is positive because it is a sum of squares. 
%Extracting the ``constant term'' does not change the positivity.

%Because $f$ is homogeneous of degree 2 in $L_3, \Sigma, \Delta$ we have
%\[
%    L_3 \frac{\partial f}{\partial L_3} + \Sigma  \frac{\partial f}{\partial \Sigma} + \Delta \frac{\partial  f}{ \partial \Delta} = 2 f
%\]

\begin{Lemma} \label{lem:Veff}
For relative equilibria with $p_i = 0$ , $i=1,\dots, 4$ define 
\[
   f(q_i, q_j, L_3) = c_0 ( q_i, q_j) + c_2(q_i, q_j) L_3^2 + O(L_3^4) \,,
\]
such that the Hamiltonian to leading order in $p$  is
\[
  H = K_{\eff} + V_{\eff} + O(p^4)
\]
where 
\begin{align*}
   K_{\eff} &= \frac{1}{2\nu_1}(p_1^2 + p_2^2 + c_2(q_3, q_4) L_3^2) + \frac{1}{2\nu_2}(p_3^2 + p_4^2 + c_2(q_1, q_2) L_3^2), \\
   V_{\eff} &= V(q_1^2+q_2^2, q_3^2+q_4^2, q_1 q_3 + q_2 q_4) +  \frac{1}{2\nu_1}c_0(q_3, q_4)  + \frac{1}{2\nu_2}c_0(q_1, q_2) \,.
\end{align*}
Then critical points of $V_{\eff}$ are relative equilibria of $H$. If in addition   $ \frac{1}{2\nu_1}  c_2(q_3, q_4)   + \frac{1}{2\nu_2} c_2(q_1, q_2)  > 0$
at the critical point, and the critical point is a minimum of $V_\eff$, then it is a minimum of $H$. 
\end{Lemma}
\begin{proof}
By construction $V_\eff$ is independent of $p$ and $\partial K_\eff /\partial p_i$ vanishes for $p_i = 0$.
Moreover $\partial K_\eff / \partial q_i $ vanishes at $p_i = 0$. Thus the remaining condition for a 
critical point is $\partial V_\eff / \partial q_i = 0$. 
The fact that there are no linear terms in momenta in $K_\eff$ follows from the fact that $f$ is an even function in $L_3$.
The second statement is about the positivity of $K_\eff$ 
as a quadratic form in $p$. If the two additional terms involving $c_2$ are positive definite in $p$
then $K_\eff$ is positive definite as a whole, since the sum of positive definite matrices is again 
positive definite. Thus if $V_\eff$ has a non-degenerate strict minimum at $q^*$ then $(q,p) = (q^*,0)$ is a minimum of 
$H$, since higher order terms corresponding to $O(p^4)$ all vanish at the equilibrium.
\end{proof}

The function $c_0$ is the constant term of the Taylor expansion $c_0(q_i, q_j) = f(q_i, q_j, 0)$.
Define  %  $\omega_i = \partial H / \partial \mu_i = \partial V_\eff / \partial \mu_i$ at the equilibrium such that 
the (effective) moments of inertia 
\[
    I_1^{-1} =  \frac{q_1^2 /\nu_2 + q_3^2 / \nu_1}{4 A^2} , \quad % \approx \frac{\mu_1}{\nu_2 q_4^2}, \quad
    I_2^{-1} =  \frac{q_2^2 /\nu_2 + q_4^2 / \nu_1}{4 A^2}  \,. % \approx \frac{\mu_2}{\nu_1 q_1^2} \,.
\]
Then the effective potential can be written as
\begin{equation} \label{eqn:Veff}
   V_\eff  %  \frac{1}{2\nu_1}  f(q_3, q_4, 0) + \frac{1}{2\nu_2} f(q_1, q_2, 0) + V
%      =   \frac{1}{ 8 A^2 \nu_1 \nu_2} ( \mu_1^2 (q_1^2 \nu_1 + q_3^2 \nu_2) + \mu_2^2 ( q_2^2 \nu_1 + q_4^2 \nu_2)) + V \,.
= \tfrac12 (  \mu_1^2 I_1^{-1} + \mu_2^2 I_2^{-1}) + V \,.
\end{equation}
The additional terms in $K_\eff$ proportional to $L_3^2$ are obtained from the Taylor series of $f$ as
\[
 %      \frac{L_3^2}{ 8 A^2 \nu_1 \nu_2(\mu_1^2 - \mu_2^2)} ( -\mu_1^2 (q_1^2 \nu_1 + q_3^2 \nu_2) + \mu_2^2 ( q_2^2 \nu_1 + q_4^2 \nu_2))  \,.
     \frac{L_3^3}{2( \mu_1^2 - \mu_2^2)} ( -  \mu_1^2 I_1^{-1} +  \mu_2^2 I_2^{-1}) \,.
\]

It appears as if in the limit $\mu_2 \to 0$ we never have a positive definite $K_\eff$. However,
notice that for small $\mu_2$ the orders in $\mu_2$ of the various $q_i$ are different, in particular $q_4$ is of order 1
while $q_1$ is of order $\mu_2^2$, and $q_2$ and $q_3$ are negligibly small. 
Also notice that the condition stated in the Lemma is sufficient for the definitness of $K_\eff$, but not necessary.

%For later use we note that $V_\eff = \tfrac12( \omega_1 \mu_1 + \omega_2 \mu_2) + V$ where 
%by definition $\omega_i = \partial H / \partial \mu_i = \partial V_\eff / \partial \mu_i$.

% Notice that no statement is made about equilibrria with non-zero momenta. 

Finally let us remark that up to this point we have not assumed any particular form of the potential, other 
than that it depends on $x_i \cdot x_j$ only. From now (with the exception of Lemma~\ref{genequcond})
we will treat the Newtonian case only.

\section{Equilibria of the reduced Hamiltonian for two equal masses}

\begin{figure}
\includegraphics[width=5cm]{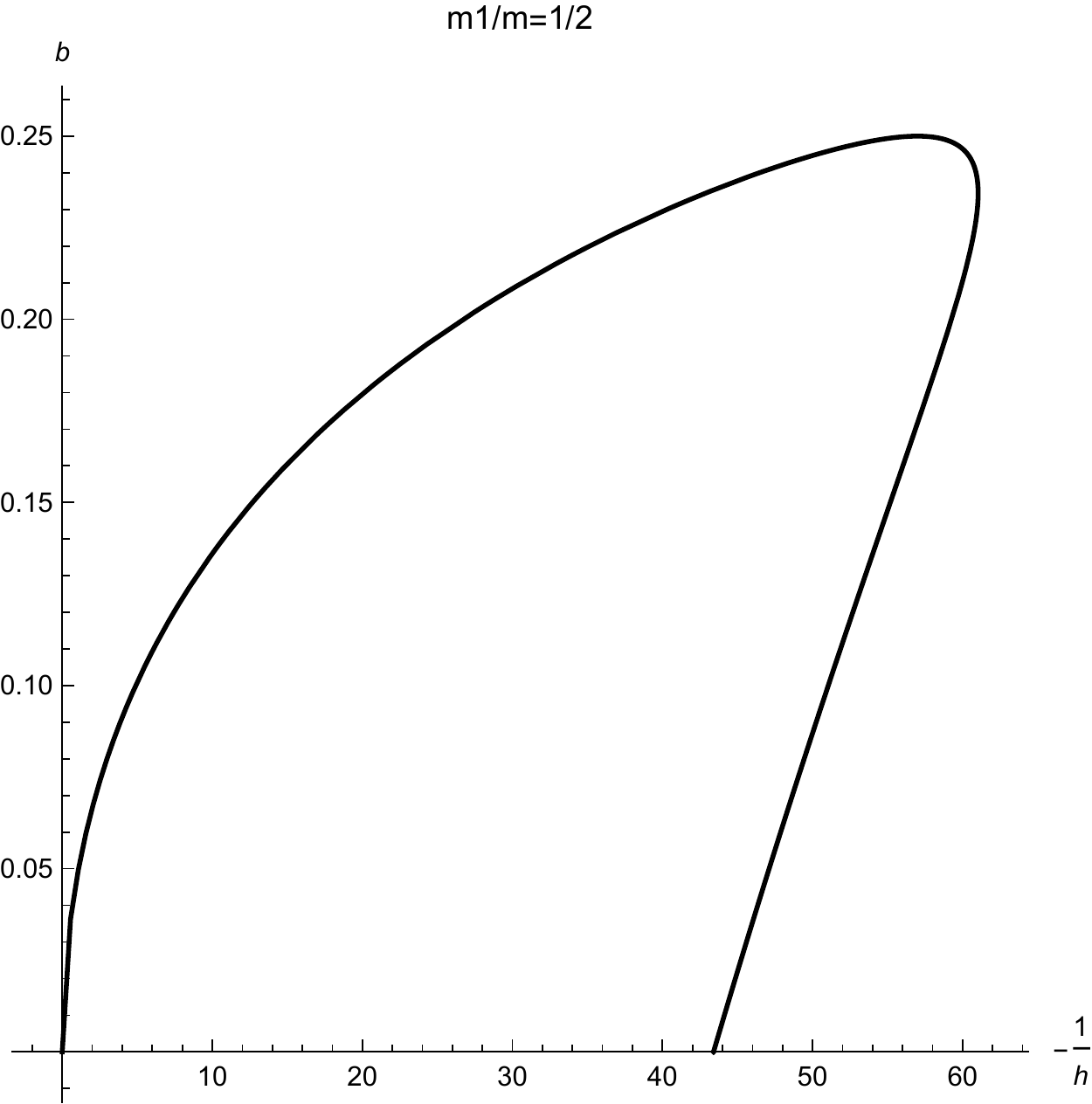}
\hspace{1cm}
\includegraphics[width=5cm]{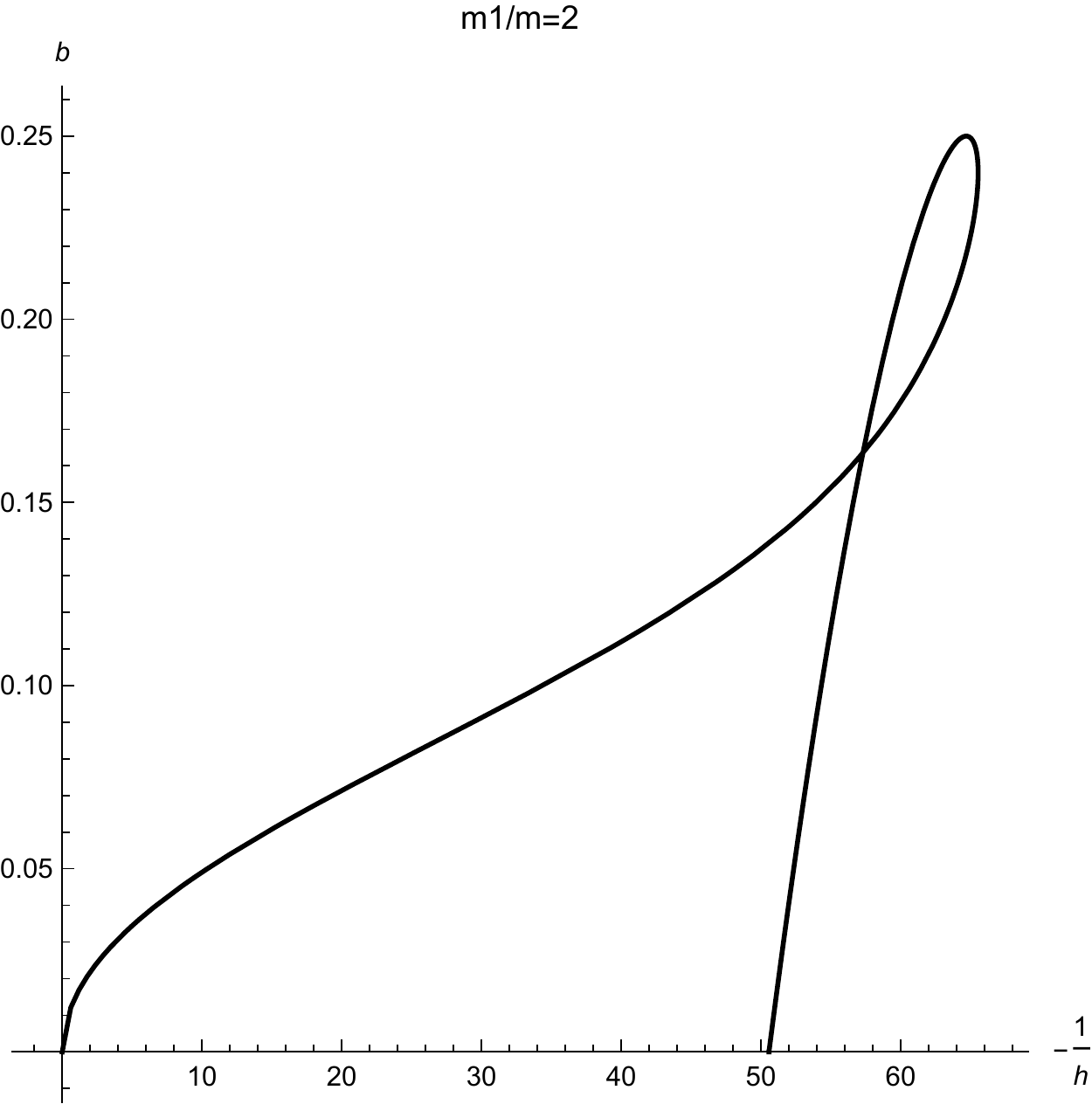}
\caption{Scaled energy-momentum diagram of the isosceles family of relative equilibria (or balanced configuration) in the 3-body problem in dimension 4 for two different mass ratios. These relative equilibria are minima of
the Hamiltonian for sufficiently large negative scaled energy $h$, which occurs for small $b$
corresponding to small $\mu_2$.} \label{fig1}
\end{figure}

Before treating the case of arbitrary masses we now discuss the case of two equal masses $m_2 = m_3 = m$.
This case is technically simpler since the equilibrium conditions for one of the equilibria can be solved explicitly. 
In the general case, instead we only have a series solution near $\mu_2  = 0$.

\begin{Theorem}
For $m_2 = m_3 = m$ an isosceles triangle  is a relative equilibrium of the 3-body problem in $\mathbb{R}^4$
for any  momenta $\mu_1 > \mu_2 > 0$.
%, that is a critical point of the reduced Hamiltonian of Theorem 2, i.e.{} its gradient is equal to zero there.
The relative equilibrium is given by  $q_2=q_3 = 0$  and 
$p_1 = p_2 = p_3 = p_4 = 0$,
and two additional equations that relate $q_1, q_4$ to $\mu_1, \mu_2$:
\[
   \frac{m^2}{q_1^2}+\frac{4 m m_1 q_1}{\left(q_1^2+4 q_4^2\right)^{3/2}}-\frac{\mu_2^2}{\nu_1 q_1^3} = 0,
   \quad
   \frac{16 m m_1 q_4}{\left(q_1^2+4 q_4^2\right)^{3/2}}-\frac{\mu_1^2  }{ \nu_2 q_4^3} = 0 \,.
\]
For $\mu_2$ sufficiently small the corresponding critical point of the reduced Hamiltonian \eqref{eqn:Hamred} is a minimum.
\end{Theorem}
\begin{proof}
In the isosceles case $a_{1} = a_{2} = \frac{1}{2}$, $\nu_1 = m/2$, and $\nu_2 = 2 m m_1 / (2 m + m_1)$.
The derivative of the function $f$ with respect to $p_i$ at vanishing momenta vanishes.
%In the equal mass isosceles case the potential has only two terms and is given by 
%\[
 %  V = -\frac{m^2}{|q_1|} - 2 \frac{ 2 m m_1}{\sqrt{ q_1^2 + 4 q_4^2}} \,.
%\]
%...
The claim that the critical point is a minimum is proved in the following Lemmas.
\end{proof}

Denote the mass ratio as $n=m_1/m$ and the length ratio as $\rho = q_1/q_4$.
To rationalise the square root use $\rho = 4 t / ( 1 - t^2)$ where $ t \in (0, 1)$
and $t = 2 - \sqrt{3}$ corresponds to the equilateral triangle. With this parametrisation
the equilibrium condition determines $\mu_i^2$ as rational functions of of $n$ and $t$
(up to scaling with $m^3 q_4$). One can check that by the implicit function theorem this 
is always possible instead of eliminating $q_1, q_4$. 

A family of equilibria is best described in an energy-momentum diagram, 
see Figure~\ref{fig1}. Define the scaled energy $h$ and dimensionless momentum $b$ as
\[
  (h,b) = \left(    (\mu_1 + \mu_2)^2 H|_{eq} ,    \frac{\mu_1 \mu_2}{(\mu_1 + \mu_2)^2}  \right) \,.
\]
In Fig.~\ref{fig1} instead $(-1/h, b)$ is plotted since we are interested in the limit
$\mu_2 \to 0$ where $h \to - \infty$ and $b \to 0$. The parameter along the curve is the shape parameter 
$\rho = q_1 / q_4$  ranging from $\rho = 0$ (collision, left endpoint with $b=0$)  through $\rho =2/\sqrt{3}$ (equilateral)  
to $\rho = \infty$ (collinear, right endpoint with $b=0$). Equivalently, the parameter $t \in (0,1)$ can be used.

In the limit $\mu_2 \to 0$ and hence $b \to 0$, while $\mu_1$ remains finite, the equilibrium condition of Theorem 3 can be written as ($n =  \frac{m_1}{m}$)
\[
   q_1 = \frac{ 2 \mu_1^2}{m^3} b^2 ( 1 + 8 b + O(b^2)), \quad
   q_4 = \frac{ \mu_1^2}{4 m^3 n^2 ( 2 + n) } ( (2 + n)^2 + 24 b^4 + O(b^5)) \,.
\]
The distance between the equal mass particles goes to zero with $b^2$, while
the distance to the third particle remains finite in this limit.

%The value of the Hamiltonian at the equilibrium is 
%\[
%    \frac{  \mu_1^2 } { 2 \nu_2 q_4^2} + \frac{ \mu_2^2}{ 2 \nu_1 q_1^2} - \frac{m^2}{q_1} - \frac{ 4 m m_1 }{ \sqrt{ q_1^2 + 4 q_4^2} }  \,.
%\]

\begin{Lemma}
The Hessian of the reduced Hamiltonian at the isosceles equilibrium is block-diagonal, 
with three non-trivial $2\times 2$ blocks and two explicit eigenvalues given by $1/\nu_i$. 
In the following expressions for these $2\times 2$ blocks  % subsequently labeled 1, 2, and 3, respectively, 
the relationship between $q_1, q_4$ and $\mu_1, \mu_2$ has not been used.
% \footnote{however, when we write $|_{eq}$ we mean restricted to $q_2 = q_3  =0$, $p_i = 0$}
\[
\left. \frac{\partial^2 H } { \partial^2 ( q_2, q_3) } \right|_{eq} = \left(
\begin{array}{cc}
 \frac{ \mu_2^2}{\nu_2 q_1^2 q_4^2  }+\frac{m^2}{q_1^3}+
    \frac{4 m m_1 \left(q_1^2-8 q_4^2\right)}{\left(q_1^2+4 q_4^2\right)^{5/2}} & \frac{ \mu_1^2}{\nu_2 q_1 q_4^3  }+\frac{
   \mu_2^2}{\nu_1 q_1^3 q_4}-\frac{48 m m_1 q_1 q_4}{\left(q_1^2+4 q_4^2\right)^{5/2}} \\
 \frac{ \mu_1^2}{\nu_2 q_1 q_4^3  }+\frac{ \mu_2^2}{\nu_1 q_1^3 q_4}-\frac{48 m m_1 q_1
   q_4}{\left(q_1^2+4 q_4^2\right)^{5/2}} & \frac{ \mu_1^2}{\nu_1 q_1^2 q_4^2}-\frac{32 m m_1
   \left(q_1^2-2 q_4^2\right)}{\left(q_1^2+4 q_4^2\right)^{5/2}} \\
\end{array}
\right)
\]

\[
\left.\frac{\partial^2 H } { \partial^2 ( q_1, q_4) }  \right|_{eq}=\left(
\begin{array}{cc}
 \frac{3 \mu_2^2}{\nu_1 q_1^4 }-\frac{2 m^2}{q_1^3}-\frac{8 m_1 \left(q_1^2-2 q_4^2\right) m}{\left(q_1^2+4 q_4^2\right)^{5/2}} & -\frac{48 m m_1 q_1 q_4}{\left(q_1^2+4 q_4^2\right)^{5/2}} \\
 -\frac{48 m m_1 q_1 q_4}{\left(q_1^2+4 q_4^2\right)^{5/2}} & \frac{3 \mu_1^2}{ \nu_2 q_4^4}
     +\frac{16 m m_1 \left(q_1^2-8 q_4^2\right)}{\left(q_1^2+4 q_4^2\right)^{5/2}} \\
\end{array}
\right)
\]

\[
\left.\frac{\partial^2 H } { \partial^2 ( p_2, p_3) } \right|_{eq} =
\frac{1}{\mu_1^2 - \mu_2^2} 
\left(
\begin{array}{cc}
  \mu_1^2 \left( \frac{1}{\nu_1}-\frac{ q_1^2}{\nu_2 q_4^2 }\right) &
           \frac{ \mu_1^2  q_1}{ \nu_2 q_4 }-\frac{ \mu_2^2 q_4}{\nu_1 q_1} \\
 \frac{ \mu_1^2 q_1}{\nu_2 q_4  }-\frac{ \mu_2^2 q_4}{\nu_1 q_1} & 
          \mu_2^2  \left( \frac{ q_4^2}{\nu_1 q_1^2} -\frac{1}{\nu_2 }\right) \\
\end{array}
\right)
\]
\end{Lemma}
\begin{proof}
The blocks are found by differentiating $V_\eff$ and $K_\eff$,  evaluated at $p_1= p_2=p_3=p_4=0$, $q_2 = q_3 = 0$.
\end{proof}

In the following $\mu_1, \mu_2$ have been eliminated using the equilibrium condition, 
parametrised by $t$.
The eigenvalues then depend (up to an overall scaling) 
on the essential parameters $n = m_1/m $ and $t$. 

\begin{Lemma}
For $ t \to 0$ all eigenvalues of the Hessian are positive. \\
The $m^2/q_4^3$-scaled eigenvalues of  the $(q_2, q_3)$-block  for $t \to 0$ are
\[
\frac{n^2}{(2 n+4) t^2}-\frac{1}{4 t}-\frac{2 \left(7 n^2+2 n\right)}{n+2}+O\left(t\right), \quad
\frac{1}{64  t^3}+\frac{\frac{17}{64}+\frac{1}{8 n}}{t}+\frac{11 n^2+6 n}{n+2}+O\left(t\right) \,.
\]
The $m^2/q_4^3$-scaled eigenvalues of the $(q_1, q_4)$-block for $t \to 0$ are
\[
\frac{1}{64 t^3}-\frac{3}{64 t}+2 n+O\left(t\right), \quad
2 n+12 n t^2+O\left(t^3\right) \,.
\]   
The $m$-scaled eigenvalues of the $(p_2, p_3)$-block for $t \to 0$ are
\[
   2-\frac{8 (3 n-2) t^2}{n}+O\left(t^3\right), \quad
   \frac{n+2}{16 n^2 t}+\frac{(n+2)^2}{32 n^4}+O\left(t\right) \,.
\]
\end{Lemma}
\begin{proof}
Because of the block-diagonal structure of the Hessian these can be obtained by solving quadratic 
equations and expanding the roots in the small parameter $t$.
\end{proof}
% Note that four eigenvalues diverge in the limit, while two have a finite limit.

Instead of using series expansion we can check conditions for the Hessian to be positive definite.
This can be done globally, for all $t \in (0, 1)$. 
Definitness is lost when the determinants of the blocks go through zero or infinity. 
The expression $ \mu_1^2 - \mu_2^2$ appears in the determinant of two blocks, once 
in the numerator and once in the denominator.
Equality occurs when the polynomial
 \[
 P_2(n,t) =  2 n^2 \left(t^4-6 t^2+1\right) \left(t^2+1\right)^2-n t \left( 64 t^3+(t^2+1)^3 \right)-2 t
   \left(t^2+1\right)^3
\]
vanishes, which is obtained from the non-trivial equilibrium conditions. 
The curve $P_2(n,t) = 0$ starts at the origin in $(n,t)$ and asymptotes to $t = \sqrt{2} - 1$, 
as does $P_1(n,t) = 0$,  the dashed line in Figure~\ref{fig2} indicates the asymptote.
When $\mu_1 = \mu_2$  the maximal value $b = 1/4$ in Figure~\ref{fig1} is reached.
Note that this case is excluded in the reduction theorem.

\begin{figure}
\includegraphics[width=9cm]{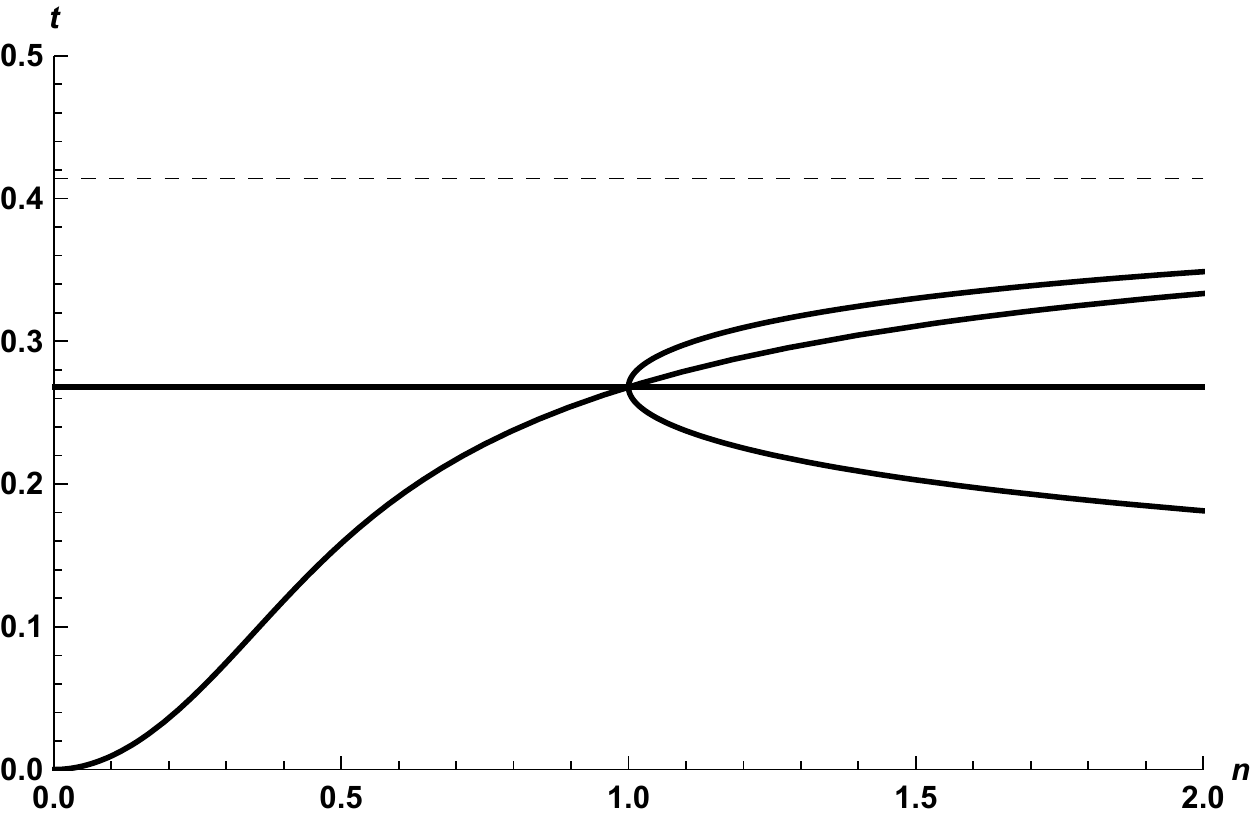}
\caption{Parameter space $n = m_1/m > 0$ and shape parameter $t \in (0,1)$ of the isosceles equilibrium.
The curves divide the positive quadrant into 6 regions. The horizontal line $t = 2 - \sqrt{3}$ corresponds to the 
equilateral triangles. The parabola-shaped curve $P_1(n,t) = 0$ indicates a vanishing of the determinant of the $(q_2, q_3)$-block.
The  curve $P_2(n,t) = 0$ starting at the origin indicates a vanishing of the determinant of the $(q_2, q_3)$-block and an infinity in the
determinant of the $(p_2,p_3)$-block. In the region adjacent to the $n$-axis all eigenvalues are positive and the 
isosceles solution is a minimum of the 3-body problem in $\mathbb{R}^4$.}
\label{fig2}
\end{figure}

\begin{Lemma}
The eigenvalues of the $(q_2, q_3)$-block are positive if $(\mu_1^2 - \mu_2^2) P_1(n,t) > 0$.

The eigenvalues of the $(q_1, q_4)$-block are always positive.

The eigenvalues of the  $(p_2, p_3)$-block are positive if $ (( 2 - \sqrt{3}) - t) ( \mu_1^2 - \mu_2^2) > 0$.
\end{Lemma}
\begin{proof}
The determinant of the $(q_2, q_3)$-block vanishes when $\mu_1^2 = \mu_2^2$ and when the polynomial
\[
   P_1(n,t) =  32 t^3 \left( 3 n ( t^4-6 t^2+1) + 2 ( t^4 - 10 t^2 + 1)\right)- ( t^2 + 1)^5
%   t^{10}-5 t^8+64 t^7-10 t^6-640 t^5-10 t^4+64 t^3-5 t^2-1
\]
vanishes.

The determinant of the $(q_1, q_4)$-block is 
\[
\frac{m^4 n \left(t^2-1\right)^6 \left(128 n t^3+t^6+15 t^4+15 t^2+1\right)}{32 q_4^6 t^3
   \left(t^2+1\right)^6}
\]
which is positive for positive $n$ and $t$.

The determinant of  the $(p_2, p_3)$-block  (without the prefactor $\mu_1^2 - \mu_2^2$) is 
\[
\frac{ (1 - t^2) \left(t^2-4 t+1\right) \left(t^4+4 t^3+18 t^2+4 t+1\right)}{2 t
   \left(t^2+1\right)^3},
\]
where only the middle factor in the numerator changes sign at $t = 2 - \sqrt{3}$ (equilateral triangle).
The prefactor itself vanishes when $\mu_1^2 = \mu_2^2$, which implies $P_2(t) = 0$.
\end{proof}
The curves $P_1(t) = 0$ and $P_2(t) = 0$ together with $t = 2 - \sqrt{3} $ are shown in Figure~\ref{fig2}.

The frequencies $\omega_i$ for rotation in the eigenplanes are determined by
differentiating the Hamiltonian with respect to $\mu_1$ and $\mu_2$. 
At the equilibrium the only contribution comes from $V_\eff$ such that 
$\omega_i = \mu_i I_i^{-1}$.
This gives
\[ 
      \omega_1  = \frac{ \mu_1}{\nu_2 q_4^2} ,  \quad   \omega_2  = \frac{ \mu_2 }{\nu_1  q_1^2 }
\]
and hence
\[
    \omega_1  =    \sqrt{ \frac{ m (2 + n) } {q_4^3}}   \sqrt{  \frac{(1 - t^2)^3 }{(1+t^2)^3} }, \quad
%    \omega_2 = \sqrt{ \frac { m} { q_4^3} }  t^{-3/2}  \sqrt{  \frac{ (1-t^2)^3}{ (1 + t^2)^3}  } \sqrt{ 1 + 3 t^2 + 32 n t^3 + 3 t^4 + t^6}
    \omega_2 = \sqrt{ \frac { 2 m} { q_1^3} }  \sqrt{ 1 +  \frac{ 32 n t^3 }{ (1 + t^2)^3}  }
\]
such that for $t \to 0$ the frequencies of rotation are given by Kepler's third law.
Frequency $\omega_2$ is determined by masses $m_2$ and $m_3$ orbiting around each other with distance $q_1$, ignoring $m_1$, 
while frequency $\omega_1$ is determined by mass $m_1$ orbiting around the combined
mass $m_2 + m_3$ at distance $q_3$, and hence behaves like $\sqrt{ M / q_4^3}$.
Note that $\omega_2$ diverges like $t^{-3/2}$, while $\omega_1$ remains finite.
The frequency ratio simply is $((1 + t^2)^2 + 32 n t^3)/(32 (2+n) t^3)$ which in general is irrational so 
that the relative equilibrium is a quasiperiodic motion with two incommensurate frequencies.

%\subsection{}

%\section{Questions/Remarks}
%
%We could start with index zero, and use it for COM.

%\section{Euler collinear}
%
%Consider equal masses  $m_3 = m_2$ and the limit $\mu_1 \to 0$ instead of $\mu_2 \to 0$.
%% and hence $q_4 \to 0$. 
%In this case the limit to the 3-dimensional case does exist and is given by 
%the Euler collinear relative equilibrium. In this case the Hamiltonian is not definite, so instead we 
%compute the eigenvalues of the linearised vector field. 
%The $8\times 8$ matrix splits into two $4\times 4$ blocks with a pair of real eigenvalues
%\[
%   \pm \sqrt{4 \sqrt{n^2+18 n+32}-4 n+6} + O(2)
%\]
%and three pairs of imaginary eigenvalues
%\[
%    \pm i \sqrt{4 \sqrt{n^2+18 n+32}+4 n-6} + O(2), \quad
%    \pm i 4 \sqrt{ 4 + 2 n} + O(2), \quad
%    \pm i \sqrt{ 2 + 8 n}  + O(2)\,.
%\]
%These are the well known eigenvalues of the Euler collinear solution in the isosceles case????
%
%% Maybe give $\mu_2 = \mu_2(n)$ and then write these as functions of $n$?
%
%
%For comparison:  In the  limit $\mu_2 \to 0$ and hence $t \to 0$ studied in the previous chapter
%the linearised vector field splits into two $4\times 4$ blocks with eigenvalues
%\[
%    \pm i \left(  (2 t)^{-3/2} \pm \sqrt{ 2+n}  + O(t^{1/2}) \right), \quad
%    \pm i  \sqrt{ 2+n} ( 1 + 3 t^2 + O(t^3)), \quad
%    \pm i \left( (2 t)^{-3/2} - \tfrac{3}{16} (2 t)^{-1/2} + O(t^{3/2})    \right)
%\]
%

\section{General masses}

Denote by $\mu$ the ratio $\mu = \mu_2/\mu_1$ and by $M = m_1 + m_2 + m_3$. 
We are now giving a series expansion of the coordinates of the equilibrium  in the limit $\mu_2 \to 0$ (and hence $\mu \to 0$).

\begin{Lemma} \label{genequcond}
The equilibrium condition $D_q V_\eff = 0$ implies  the solvability condition $q_1 q_2 \nu_1 + q_3 q_4 \nu_2 = 0$.
Simplifying the equilibrium condition using the solvability condition gives
\begin{align*}
   \frac{ I_1 \mu_2^2 q_4}{8 A^3 \nu_1 \nu_2} & =  2 q_1 V_1 + q_3 V_3, \quad
   \frac{ -I_2 \mu_1^2 q_3}{8 A^3 \nu_1 \nu_2} =  2 q_2 V_1 + q_4 V_3, \\
   \frac{ -I_1 \mu_2^2 q_2}{8 A^3 \nu_1 \nu_2} & =  2 q_3 V_2 + q_1 V_3, \quad
   \frac{ I_2 \mu_1^2 q_1}{8 A^3 \nu_1 \nu_2} =  2 q_4 V_2 + q_2 V_3 \,.  
\end{align*}
\end{Lemma}
\begin{proof}
After reduction the potential is a function of the form $ V(q_1^2 + q_2^2, q_3^2+q_4^2, q_1 q_3 + q_2 q_4)$. 
Thus the gradient with respect to $q$ is $V_q = ( 2 q_1 V_1 + q_3 V_3, 2 q_2 V_1 + q_4 V_3, 2 q_3 V_2 + q_1 V_3, 2 q_4 V_2 + q_2 V_3)$,
where $V_i$ denotes the derivative of $V$ with respect to its $i$th argument. Reading the right hand side as 
a linear equation in $V_i$, $i=1,2,3$ the solvability condition is that the left hand side is orthogonal to the kernel of the transpose 
of the matrix of the linear system. The kernel is given by $(-q_2, q_1, -q_4, q_3)^t$ and is equal to the derivative of $L_3$ with respect to $p$.
The solvability condition is $(\mu_1^2 - \mu_2^2) ( q_1 q_2 \nu_1 + q_3 q_4 \nu_2)/ ( 4 A^2 \nu_1 \nu_2) = 0$, 
which proves the stated solvability condition. Using the solvability condition the moments of inertia simplify to 
\[
   I_1 = \nu_2 q_4^2  + \nu_1 q_2^2 , \quad 
   I_2 = \nu_1 q_1^2  + \nu_2 q_3^2  \,,
\]
and this leads to the stated left hand side of the equilibrium condition.
\end{proof}

The previous Lemma is valid for an arbitrary potential depending on $x_i \cdot x_j$ only. 
From now on all statements are about the Newtonian case.

\begin{Lemma} \label{geneq}
A critical point with $p_i = 0$, $i = 1, \dots, 4$ of the reduced Hamiltonian \eqref{eqn:Hamred} for general masses has
a power series expansion for small $\mu$ given by
%\begin{align*}
%\frac{q_1}{\mu_1^2} = &
%\frac{\mu ^2 \left(m_2+m_3\right)}{m_2^2 m_3^2} 
%- \frac{\mu ^8 m_1^7 \left(m_2+m_3\right)^9}{M^3 m_2^8 m_3^8}
%+O(\mu^{12}) % note that from 12 we have all even power
%\\
%\frac{q_2}{\mu_1^2} = & 
%\frac{3 \mu ^{10} m_1^9 \left(m_2-m_3\right) \left(m_2+m_3\right)^{11}}{2 M^4 m_2^{10} m_3^{10}} 
%-\frac{3 \mu ^{14} m_1^{13} \left(m_2-m_3\right) \left(m_2+m_3\right)^{15} \left(5 m_2^2+24 m_3 m_2+5 m_3^2\right)}{8 M^6 m_2^{14} m_3^{14}}
%+O(\mu^{16})
%\\
%\frac{q_3}{\mu_1^2} = &    
%-\frac{3 \mu ^{12} m_1^{10} \left(m_2-m_3\right) \left(m_2+m_3\right)^{12}}{2 M^4 m_2^{11} m_3^{11}}
%+\frac{15 \mu ^{16} m_1^{14} \left(m_2-m_3\right) \left(m_2+m_3\right)^{16} \left(m_2^2+6 m_3 m_2+m_3^2\right)}{8 M^6 m_2^{15} m_3^{15}} 
%+O(\mu^{18})
%\\
%\frac{q_4}{\mu_1^2} = & 
%  \frac{M}{m_1^2 \left(m_2+m_3\right)^2}
%  + \frac{3 \mu ^4 m_1^2 \left(m_2+m_3\right)^2}{2 M m_2^3 m_3^3}
%+O(\mu^{8}) % note that from 8 we have all even powers
%\end{align*}
\begin{align*}
\frac{q_1}{\kappa \mu_1^2} & =  u^2 - \frac{ m_1}{m_2 + m_3} u^8 + O(u^{12})
\\
\frac{q_2}{\kappa \mu_1^2} & = 
\frac{3 u ^{10} m_1 (m_2-m_3) }{2 (m_2+m_3)^2}  \left( 1
-\frac{u ^4  (5 m_2^2+24 m_3 m_2+5 m_3^2)}{4 (m_2+m_3)^2} \right)
+O(u^{16})
\\
\frac{q_3}{\kappa \mu_1^2} & =   
-\frac{3 u ^{12} M m_2 m_3 (m_2-m_3) }{2 (m_2+m_3)^4} \left(1 
- \frac{ 5 u^4 (m_2^2+6 m_3 m_2+m_3^2)}{4 (m_2+m_3)^2}  \right)
+O(u^{18})
\\
\frac{q_4}{\kappa \mu_1^2} & =  1 + \frac{ 3 u^4 m_2 m_3}{2 (m_2+m_3)^2} 
+O(u^{8}) % note that from 8 we have all even powers
\end{align*}
where $\kappa = \frac{M}{m_1^2 (m_2 + m_3)^2} = \frac{ \nu_1}{\nu_2 m_1m_2 m_3}$  and
$u = \mu / (m_2 m_3 \sqrt{ \kappa /(m_2+m_3)})$.
%
% The exact relation $q_1 q_2 \nu_1 + q_3 q_4 \nu_2 = 0$ holds.
\end{Lemma}
\begin{proof}
The equilibrium condition $\partial V_{\mathrm{\eff}}/\partial q_i = 0$  in the limit $\mu \to 0$
 has only $q_4$ with a non-vanishing limit.  
 The leading orders of $q_1, q_2, q_3$ in $\mu$ are $2, 10, 12$, respectively.
 This balances the leading order of 3 of the 4 equilibrium conditions of Lemma~\ref{genequcond}.
 However, the condition $\partial V_\eff / \partial q_2 = 0$ is not balanced at leading 
 order but determines two higher order coefficients. %  of the expansion of $q_1$ and $q_4$.
The remaining higher order coefficients of the power series solutions are then determined
 order by order.  A natural dimensionless expansion parameter is $u$ as determined by the leading 
 order coefficients of $q_1$ and $q_4$.
 %Reading the equilibrium condition as a linear system for the derivatives of the three terms in the potential 
 %gives a linear system whose solvability condition is $q_1 q_2 \nu_1 + q_3 q_4 \nu_2 = 0$. 
\end{proof}

The surprisingly high powers of the leading order in $\mu$ (or $u$) for $q_2$ and $q_3$ can be interpreted 
as saying that in the collision limit $\mu_2 \to 0$ the configuration is approximately isosceles. Of course 
for $m_3 = m_2$ the solution is exactly isosceles and $q_2 = q_3 = 0$.
The distances  between the particles are  
\[
  (||r_2-r_3||, ||r_3  - r_1||,||r_1-r_2||) =\kappa \mu_1^2 \left( u^2, 1 + u^4 m_2\frac{m_2 + 3 m_3}{2 (m_2+m_3)^2}, 1 +  u^4 m_3\frac{ 3 m_2 + m_3}{2 (m_2+m_3)^2}\right)   + O(u^8)
\]
and the scalar products are 
\[
  ( ||x_1||^2 , ||x_2||^2 ,x_1 \cdot x_2)  = \kappa^2 \mu_1^4\left( u^4 + O(u^{10}), 1 + O( u^4),  \frac{ 3 m_1 (m_2-m_3)}{ 2 (m_2+m_3)^2} u^{10} + O(u^{14}) \right) \,.
\]
The area behaves like $A = \tfrac12 u^2 \kappa^2 \mu_1^4 + O(u^6)$.

\begin{Theorem}
The relative equilibrium of the 3-body problem in 4-dimensional space given in Lemma~\ref{geneq}
is a minimum of the reduced Hamiltonian \eqref{eqn:Hamred}.
\end{Theorem}
\begin{proof}
The reduced Hamiltonian is the sum of kinetic and potential energy.
The effective potential \eqref{eqn:Veff} has a minimum at this equilibrium, as is shown in the next Lemma.
We now show that the effective kinetic energy $K_\eff$ is positive definite for sufficiently small $\mu$.
The coefficient of the correction term proportional to $L_3^2$ in $K_\eff$ is
\[
   -\frac{\mu_1^2}{I_1} + \frac{\mu_2^2}{I_2} = 
   -\frac{\mu_1^2}{\nu_2 q_4^2 + \nu_1 q_2^2} + \frac{\mu_2^2}{\nu_1 q_1^2 + \nu_2 q_3^2} =
   -\frac{\mu_1^2}{\nu_2 q_4^2} + \frac{\mu_2^2}{\nu_1 q_1^2}  + O(\mu_2^{10})
\]
and since $q_1 = O(\mu_2^2)$ the second term dominates for $\mu_2 \to 0$
while the first (negative) term is $O(1)$
and so the coefficient is positive for sufficiently small $\mu_2$.
By Lemma~\ref{lem:Veff} the Hessian of the Hamiltonian with respect to $(p_1, p_2, p_3, p_4)$ is thus positive definite 
for sufficiently small $\mu_2$. Together with the following Lemma on the positivity of the Hessian of $V_\eff$ this 
implies that the critical point is a minimum of $H$.
\end{proof}
We remark that $K_\eff$ ceases to be positive definite for $\mu_2 \to \mu_1$.
We also remark that the moments of inertia are the non-zero  eigenvalues of the original inertia tensor,
which is similar to the inertia tensor $q_{12} q_{12}^t \nu_1 + q_{34} q_{34}^t \nu_2$
and using the identity $q_1 q_2 \nu_1 + q_3 q_4 \nu_2 = 0$ gives the above moments of inertia.

\begin{Lemma}
The scaled eigenvalues of the Hessian of the effective potential \eqref{eqn:Veff}  evaluated at the equilibrium  of Lemma~\ref{geneq} 
have Laurent expansions given by
\begin{align*}
 &\frac{ m_1(m_2+m_3)}{m_2 m_3} + O(u^4) \\ %  \frac{ 3 m_1 u^4}{2 (m_2 + m_3)} + O(u^6)\\
 &\frac{ m_1^2(m_2+m_3)^3}{ m_2^2 m_3^2 M u^4} - \frac{1}{u^2} + O(u^0) \\
 &\frac{ 1}{u^6} 
 % + \frac{2m_2m_3(m_2+m_3)  + m_1( 2 (m_2+m_3)^2 + 11 m_2 m_3) }{2m_1(m_2+m_3)^2 u^2} 
+ \left( 1 + \frac{11 m_2m_3}{2 (m_2+m_3)^2} + \frac{m_2m_3}{m_1(m_2+m_3)} \right) \frac{1}{u^2}  + O(u^0)  \\
 &\frac{ 1}{ u^6 } + \frac{ 9 m_2 m_3}{2 (m_2 + m_3)^2 u^2} + \frac{ 7 m_1}{m_2 + m_3} + O(u^1)
\end{align*}
% where $u = \mu \frac{m_1 (m_2+m_3)^{3/2}}{  \sqrt{M} m_2 m_3 }$ 
with an overall scaling factor of
$m_2 m_3 / q_4^3$ removed. 
\end{Lemma}
%\footnote{Instead of computing the eigenvalues in order to prove positivity we can also just consider a single 
%element and the determinant of the $2\times 2$ blocks from the approximate block-diagonalisation.}
Note that these formulas reduce to the isosceles case $m_3 = m_2$ using the relationship $u^2 = 4 t + 28 t^3 + 128 n t^4 + O(t^5)$.
As in the isosceles case, three of the eigenvalues diverge to positive infinity in the limit.

\begin{proof}
The Hessian of the effective potential evaluated at the equilibrium condition does not block-diagonalise as
in the isosceles case. 
The Hessian can be simplified using $q_1 q_2 \nu_1 + q_3 q_4 \nu_2 = 0$ and the result is
\[
\frac{ \mu_1^2 \nu_1^3 q_1^4  }{ \nu_2 I_2^3  q_4^4} \begin{pmatrix} 
 \frac{q_2^2 \nu_1 }{\nu_2 } & -q_3 q_4 & q_2 q_4 & q_2 q_3 \\
 -q_3 q_4 & 3 q_3^2 & q_1 q_4 & -3 q_1 q_3 \\
 q_2 q_4 & q_1 q_4 & \frac{q_4^2 \nu_2 }{\nu_1 } & -q_1 q_2 \\
 q_2 q_3 & -3 q_1 q_3 & -q_1 q_2 & 3 q_1^2 
\end{pmatrix}
+ \frac{ \mu_2^2 \nu_2^3 q_4^4  }{ \nu_1 I_1^3  q_1^4} \begin{pmatrix} 
 3 q_4^2 & -q_3 q_4 & -3 q_2 q_4 & q_2 q_3 \\
 -q_3 q_4 & \frac{q_1^2 \nu_1 }{\nu_2 } & q_1 q_4 & q_1 q_3 \\
 -3 q_2 q_4 & q_1 q_4 & 3 q_2^2 & -q_1 q_2 \\
 q_2 q_3 & q_1 q_3 & -q_1 q_2 & \frac{q_3^2 \nu_2 }{\nu_1 } 
\end{pmatrix}
+ D_q^2 V
\]
The prefactors of the  first two matrices are of order $-4$ and $-6$ in $\mu$, respectively. 
Considering all terms that do not involves $q_2$ or $q_3$ in the first two terms
gives the terms denoted by $a_{ij}^k$ (except $a_{22}^{-3}$) of order $2k$ in $\mu$ where $k$ ranges from 
$-3$ to 0. All other terms that involves $q_2$ or $q_3$  are of order at least 4 in $\mu$.
The Hessian $D_q^2 V$ is diagonally dominant for small $\mu$ with terms of order $-6$ and 0 in $\mu$ in the diagonal,
and these terms are also included in $a_{ii}^k$. All off-diagonal terms in $D_q^2V$ are at least of order $\mu^2$.
Thus the Hessian can be decomposed as
\[ D = D_a + D_b = 
  \begin{pmatrix}
   a_{11}^{-3} & 0 & 0 & 0 \\
 0 & a_{22}^{-3} & a_{23}^{-2} & 0 \\
 0 & a_{23}^{-2} & a_{33}^{-2} & 0 \\
 0 & 0 & 0 & a_{44}^0 \\
\end{pmatrix}
+
  \begin{pmatrix}
   0 & b_{12}^1 & b_{13}^2 & b_{14}^1 \\
 b_{12}^1 & 0 & 0 & b_{24}^2 \\
 b_{13}^2 & 0& 0 & b_{34}^3 \\
 b_{14}^1 & b_{24}^2 & b_{34}^3 & 0 \\
\end{pmatrix}
\]
where $a_{ij}^k$, $b_{ij}^k$ denote entries of leading order $2k$ in $\mu$.
Extracting the overall leading oder $\mu^{-6}$ from $D$ makes all entries power series in $\mu$.
So the symmetric $\mu^6 D$ is a perturbation of the symmetric $\mu^6 D_a$, 
and by standard theory, see, e.g.~\cite[chapter II,  \S 2.3]{Kato},
the eigenvalues of $\mu^6 D_a$ only change at order $\mu^8$, the leading order
of $\mu^6 D_b$.

%Since $D_a$ is block-diagonal and symmetric it can easily be diagonalised by an 
%orthogonal transformation $T$.
%Applying the same transformation to all of $D$ gives
%\[
%  \mu^{-6} T^t D T = D_c + D_d
%\]
%where $D_c  = \mu^{-6} T^t D_a T$ is diagonal with eigenvalues of leading order 0, 0, 1, 3 in $\mu^2$, respectively.
%The main observation is that all entries of $D_d = \mu^{-6} T^t D_b T$ are of order 4 in $\mu^2$ or higher.
%Thus $D_d$ is a perturbation of $D_c$, and by standard theorems, see, e.g. \cite{Kato}, the eigenvalues 
%of $D_c$ are only change at order $\mu^8$ by adding $D_d$. In particular the fact that the eigenvalues 
%of $D_a$ are positive for small $\mu$ remains true for $D$ as well.

The entries $a_{11}^{-3}$ and $a_{44}^0$ give two eigenvalues up to order 6 in $\mu$.
Expanding the eigenvalues of the middle $2\times 2$ block of $D_a$ gives the leading order terms 
of the other two eigenvalues as stated in the Lemma.
This shows that the eigenvalues of the Hessian of the effective potential are positive for small $\mu$.

\end{proof}

The frequencies of rotation in the cyclic angles $\theta_i$ are obtained by differentiating the reduced Hamiltonian 
with respect to $\mu_i$. At the equilibrium the only contribution comes from $V_\eff$ and hence
\[
    \omega_1 = \frac{\mu_1}{\nu_2 q_4^2 + \nu_1 q_2^2}  \approx \frac{\mu_1}{\nu_2 q_4^2}, \quad
    \omega_2 =\frac{\mu_2}{\nu_1 q_1^2 + \nu_2 q_3^2}  \approx \frac{\mu_2}{\nu_1 q_1^2} \,.
\]
Expanding  gives
\begin{align*}
   \omega_1 &= \sqrt{ \frac{M}{q_4^3} } \left( 1 - \frac{3 m_2 m_3}{4 (m_2 + m_3)^2} u^4  + O(u^8) \right)  \\
   \omega_2 &= \sqrt{ \frac{m_2+m_3}{q_1^3} } \left( 1+ \frac{m_1}{2 (m_2+m_3)} u^6 + O(u^{10}) \right) 
\end{align*}
These formulas allow for the following nice interpretation of the three-dimensional limit $u, \mu, \mu_2 \to 0$:
The mass $m_1$ encircles the colliding binary pair with frequency $\omega_1$ at 
distance $q_4$  such that  $\omega_1^2 q_4^3 = M + O(u^4)$  is  constant at 
leading order, which is Keplers 3rd law for $m_1$ encircling $m_2 +m_3$.
The binary pair of masses $m_2, m_3$ 
has diverging frequency $\omega_2$ and vanishing distance $q_1$ 
such that $\omega_2^2 q_1^3 = m_2 + m_3 + O(u^6)$ 
is constant at leading order, which is again Keplers 3rd law for $m_2$ and $m_3$ encircling each other at distance $q_1$.
In general the two frequencies are incommensurate, and hence this is a quasi-periodic relative equilibrium.

The value $h$ of the scaled value of the Hamiltonian as a function of the dimensionless angular momentum $b$ is 
\[
      h = -\frac{m_2^3 m_3^3}{ 2 (m_2 + m_3) b^2} \left( 1 - 2b + b^2 - \frac{2b^2m_1^3(m_2+m_3)^4}{Mm_2^3m_3^3}+ O(b^3) \right)
\]
The quadratic behaviour of $-1/h$ as a function of $b$ near the origin can clearly be seen in Fig.~\ref{fig1} near the origin.

There are two additional similar such solutions obtained by exchanging the masses. The formulas 
for the critical point and its eigenvalues are 
symmetric in $m_2$ and $m_3$ (except for $q_2$ and $q_3$, which flip their signs).
The fact that  $m_2, m_3$ are singled out is a result of the choice of Jacobi coordinates in the translation reduction.
The two alternative choices defining $x_1$ as either $r_3 - r_1$ or as $r_1 - r_2$ leads to two more 
solutions that are related by permuting the masses.
All three solutions limit to collision solutions, but their precise asymptotic behaviour is different
depending on the values of the masses.
When all masses are equal there is only a single family of such solutions.
When two masses are equal there are two families, one of which is shown in Fig.~\ref{fig1}.
For distinct masses there are three families with different limiting values 
for $hb^2$ given by 
$(m_i+m_j)/( m_i m_j)^3$ for each pair of indices $i,j$.

It is interesting to note that these limiting values of $h b^2$  are exactly the critical values at which 
a bifurcation of the energy surface takes place at infinity, see \cite{Albouy93}. 
In fact, there are some remarks in Albouy's paper about higher spatial dimensions \cite[section B4]{Albouy93}.
Considering $hb^2$ in the limit $\mu_2 \to 0$ has the same order as $H \mu_2^2$, and 
hence is the correct scaling invariant combination in the 3-dimensional limit.
In our analysis we found $q_4$ finite and $q_1 \to 0$. By rescaling, instead one 
can consider $q_4 \to \infty$ and $q_1$ non-zero. This corresponds to the bifurcation at infinity.
It would be interesting to analyse bifurcations at infinity in the $n$-body problem 
in general from the point of view of higher spatial dimensions.

% It would be interesting to describe the global behaviour of these three families of relative equilibria.

% Watch out for symbol $m$, for isosceles is it $m=m_2 = m_3$, but later I use $m = m_1+m_2+m_3$, just use $M$ instead!

% \section{to do}

 % include the proof of invariant set reduction ("Dirac")
 
%  we could check the eigenvalues of the linearisation by considering the Floquet problem in  unreduced space
 
% Proofs of the finals Lemmas in the isosceles case, use Veff for computation of Hessians
% 
% Insert some words about the kinetic energy positivity.
% 
% Cite Albouy 93
% 
% The limiting case $t \to 1$ limits to the three-dimensional collinear Euler configurations.  In the isosceles case the limit is $q_4 \to 0$ which implies that $\mu_1 \to 0$.
 
%  In the general (non-isosceles) case it may be useful to introduced the polar angle.

\section{Acknowledgements}

HRD would like to thank  J\"urgen Scheurle and the Fakult\"at f\"ur Mathematik at the Technische Universit\"at M\"unchen for their 
hospitality during his sabbatical in 2018.
HRD  would like to thank Alain Albouy, Alain Chenciner, Rick Moeckel and James Montaldi for extensive discussions 
at the Observatory in Paris in 2015 when the existence of minima of the 3-body problem in $\mathbb{R}^4$
was conceived. The preprint \cite{AD19} describes some of  the results of these discussions.

\bibliographystyle{amsalpha}      % mathematics and physical sciences
%\bibliography{../../bib_cv/all,../../bib_cv/hd}{}   % name your BibTeX data base
%\bibliographystyle{plain}

\input{ReductionR4arxiv.bbl}

\end{document}

%% file: ReductionR4arxiv.bbl
\def\cprime{$'$}
\providecommand{\bysame}{\leavevmode\hbox to3em{\hrulefill}\thinspace}
\providecommand{\MR}{\relax\ifhmode\unskip\space\fi MR }
\providecommand{\MRhref}[2]{%
  \href{http://www.ams.org/mathscinet-getitem?mr=#1}{#2}
}
\providecommand{\href}[2]{#2}